\documentclass[a4paper,10pt]{article}
\usepackage[utf8x]{inputenc}
\usepackage{tracefnt,amsmath,tabu,array}
\usepackage{amssymb,graphicx,setspace,amsfonts,amsbsy}
\usepackage{pifont,latexsym,ifthen,amsthm,rotating,calc,textcase,booktabs,cancel,slashed}

\newtheorem{theorem}{Theorem}[section]
\newtheorem{lemma}[theorem]{Lemma}
\newtheorem{corollary}[theorem]{Corollary}
\newtheorem{proposition}[theorem]{Proposition}
\newtheorem{definition}[theorem]{Definition}

\newtheorem{remark}[theorem]{Remark}
\newcommand{\filledbox}{\leavevmode
  \hbox to.77778em{%
  \hfil\vbox to.675em{\hrule width.6em height.6em}\hfil}}

\newcommand{\Rm}{{\mathbb R}}

\newcommand{\eps}{\varepsilon}

\begin{document}
\tabulinesep=1.0mm
\title{Weakly non-radiative radial solutions to 3D energy subcritical wave equations\footnote{MSC classes: 35L05, 35L71.}}

\author{Ruipeng Shen\\
Centre for Applied Mathematics\\
Tianjin University\\
Tianjin, China}

\maketitle

\begin{abstract}
 In this work we consider the energy subcritical 3D wave equation $\partial_t^2 u - \Delta u = \pm |u|^{p-1} u$ and discuss its (weakly) non-radiative solutions, i.e. the solutions defined in an exterior region $\{(x,t): |x|>|t|+R\}$ with $R\geq 0$ satisfying 
 \[
  \lim_{t\rightarrow \pm\infty} \int_{|x|>|t|+R} \left(|\nabla u(x,t)|^2 + |u_t(x,t)|^2\right) dx = 0.
 \]
It has been known that any radial weakly non-radiative solution to the linear wave equation is a multiple of $1/|x|$. In addition, any radial weakly non-radiative solutions $u$ to the energy critical wave equation must possess a similar asymptotic behaviour, i.e. $u(x,t)\simeq C/|x|$ when $|x|$ is large. In this work we give examples to show that radial weakly non-radiative solutions to energy subcritical equation ($3<p<5$) may possess a much different asymptotic behaviour. However, a radial weakly non-radiative solution $u$ with initial data in the critical Sobolev space $\dot{H}^{s_p}\times \dot{H}^{s_p-1}(\Rm^3)$ must coincide with a $C^2$ solution $W$ to the elliptic equation $-\Delta W = -|W|^{p-1} W$ so that $u(x,t) \equiv W(x) \simeq C/|x|$ when $|x|$ is large.  
\end{abstract}
\section{Introduction and Main Results}

\subsection{Background and topics}

The channel of energy plays an important role in the study of radial wave equation in recent years. This method is first considered in 3-dimensional case in Duyckaerts-Kenig-Merle \cite{tkm1} and then in 5-dimensional case in Kenig-Lawrie-Schlag \cite{channel5d}. Its application includes all of the following: the proof of solution resolution conjecture of energy critical wave equation with radial data in 3-dimensional case by Duyckaerts-Kenig-Merle \cite{se} and in odd dimensions $d\geq 5$ by the same authors \cite{oddhigh}; the scattering of radial, bounded solutions to 3-dimensional wave equations in energy supercritical case by Duyckaerts-Kenig-Merle \cite{dkm2} and in energy subcritical case by Shen \cite{shen2}, and many more. Let us make a brief review on the energy channel property of solutions to wave equations. 

\paragraph{Linear equation} Assume that the dimension $d\geq 3$ is odd. Let $u$ be a solution to the free wave equation (not necessarily radially symmetric):
\[ 
 \partial_t^2 u - \Delta u = 0, \quad (x,t) \in \Rm^d \times \Rm
\]
Then by Duyckaerts-Kenig-Merle \cite{dkmnonradial} we have
\[
  \lim_{t\rightarrow +\infty} \int_{|x|>|t|} |\nabla_{x,t} u(x,t)|^2 dx  + \lim_{t\rightarrow -\infty} \int_{|x|>|t|} |\nabla_{x,t} u(x,t)|^2 dx= \int_{\Rm^d} |\nabla_{x,t} u(x,0)|^2 dx. 
\]
Here $\nabla_{x,t} u = (\nabla_x u, u_t)$. Thus the only possible non-radiative solution, i.e. the free wave satisfying
\[
 \lim_{t\rightarrow \pm \infty} \int_{|x|>|t|} |\nabla_{x,t} u(x,t)|^2 dx = 0,
\]
must be zero solution. Weakly non-radiative solutions to free wave equation in the radial setting, i.e. radial free waves satisfying 
\begin{equation} \label{weakly non radiative temp}
 \lim_{t\rightarrow \pm \infty} \int_{|x|>|t|+R} |\nabla_{x,t} u(x,t)|^2 dx = 0
\end{equation} 
for a positive constant $R$ are also well-understood. Let us first introduce a few notations, which will be used through this work, before we give the results proved by Kenig et al. \cite{channel}. Fix $R \geq 0$. We define $\mathcal{H}_{R}$ to be the space consisting of the restrictions of all $\dot{H}^1(\Rm^d) \times L^2(\Rm^d)$ functions to the exterior region $\{x: |x|>R\}$, with the norm
\begin{align*}
 \|(u_0,u_1)\|_{\mathcal{H}_{R}} = \inf \left\{\|(u'_0, u'_1)\|_{\dot{H}^1(\Rm^d) \times L^2(\Rm^d)}: (u'_0(x), u'_1(x))=(u_0(x), u_1(x)), |x|>R\right\}.
\end{align*}
We also define $\mathcal{H}_{rad, R}$ to be the subspace of $\mathcal{H}_{R}$ consisting of radial functions. Because the choice 
\[
 (u'_0(x), u'_1(x)) = \left\{\begin{array}{ll}(u_0(x), u_1(x)), & |x|>R; \\ (u_0(R), 0), & |x|\leq R\end{array}\right.
\]
clearly minimize $\|(u'_0, u'_1)\|_{\dot{H}^1(\Rm^d) \times L^2(\Rm^d)}$ in the definition of $\mathcal{H}_R$ norm above if $(u_0,u_1)$ are radial, we know in the radial case $\|(u_0,u_1)\|_{\mathcal{H}_{rad,R}} = \|(\nabla u_0, u_1)\|_{L^2(\{x:|x|>R\})}$. Thus $\mathcal{H}_{rad,R}$ is a Hilbert space with pairing 
\[
 \langle (u_0,u_1), (\tilde{u}_0, \tilde{u}_1)\rangle_{\mathcal{H}_{rad, R}} = \int_{|x|>R} (\nabla u_0 (x)\cdot \nabla \tilde{u}_0(x) + u_1(x) \tilde{u}_1(x)) dx. 
\] 
It was proved in Kenig et al. \cite{channel} that a radial free wave satisfies \eqref{weakly non radiative temp} if and only if the restriction of its initial data to the exterior region $\{x: |x|>R\}$ is contained in a $(d-1)/2$-dimensional subspace $P(R)$ of $\mathcal{H}_{rad,R}$:
\begin{align*}
 P(R)  =\hbox{Span} \left\{(r^{2k_1-d}, 0), (0,r^{2k_2-d}): k_1,k_2 \in \mathbf{N}, k_1<\frac{d+2}{4}, k_2<\frac{d}{4}\right\}
\end{align*}
We use the notation $\Theta_k, k=1,2,\cdots,(d-1)/2$ for the generators of $P(R)$ given above. Here the lower index $k$ is assigned by the identity $\|\Theta_k\|_{\mathcal{H}_{rad,R}} = c_k/R^{k-1/2}$. In addition, for any radial free waves $u$ we have
\[
 \sum_{\pm} \lim_{t\rightarrow \pm \infty} \int_{|x|>|t|+R} |\nabla_{x,t} u(x,t)|^2 dx = \|\Pi_{P(R)^\perp} (u(\cdot,0), u_t(\cdot,0))\|_{\mathcal{H}_{rad,R}}^2. 
\]
Here $\Pi_{P(R)^\perp}$ is the orthogonal projection in $\mathcal{H}_{rad,R}$ on the orthogonal subspace $P(R)^\perp$. Please note that $P(R)$ is a one-dimensional space with generator $(1/r, 0)$ if $d=3$.
 
\paragraph{Energy critical nonlinear equation} The weakly non-radiative solutions to energy critical, focusing wave equation in all odd dimensions $d\geq 3$
\begin{equation} 
 \partial_t^2 u - \Delta u = + |u|^{\frac{4}{d-2}} u \label{focusing energy critical}
\end{equation} 
has also been discussed by Duyckaerts, Kenig and Merle. In summary we have (Please see \cite{se} for 3-dimensional case and \cite{oddnonradiative} for higher dimensional case $d \geq 5$)
 
\begin{theorem}
 Let $R_0>0$ and $u$ be a radial solution to \eqref{focusing energy critical} defined on the exterior region\footnote{Please see \cite{oddnonradiative} for the definition of a solution in an exterior region} $\{(x,t): |x|>R_0+|t|\}$ such that 
\[
 \lim_{t\rightarrow \pm \infty} \int_{|x|>|t|+R_0} |\nabla_{x,t} u(x,t)|^2 dx = 0. 
\]
Then there exists $R_1 \gg 1$, $l \in \Rm$ and a generator $\Theta_k$ of the spaces $P(R)$ given above, with $l \neq 0$ if $k<\frac{d-1}{2}$,  so that for all $t$ we have
\[
 \|(u(\cdot,t), u_t(\cdot,t)) - l\Theta_k \|_{\mathcal{H}_{rad,R}} \lesssim \max\left\{\frac{1}{R^{(k-\frac{1}{2})\frac{d+2}{d-2}}}, \frac{1}{R^{k+\frac{1}{2}}}\right\}, \quad \forall R>R_0+|t|. 
\]
In addition, if $k = \frac{d-1}{2}$(when $d=3$ this is always ture), then $u$ coincides with a stationary solution for large $r$. 
\end{theorem}

\noindent It immediately follows that we have 
\begin{equation}
 \|\Pi_{P(R)^\perp} (u(\cdot,t), u_t(\cdot,t))\|_{\mathcal{H}_{rad,R}} \lesssim \max\left\{\frac{1}{R^{(k-\frac{1}{2})\frac{4}{d-2}}}, \frac{1}{R}\right\} \|\Pi_{P(R)} (u(\cdot,t), u_t(\cdot,t))\|_{\mathcal{H}_{rad,R}}. \label{fall in PR}
\end{equation} 
Namely the data of any weakly non-radiative solutions to \eqref{focusing energy critical} ``almost'' fall in the space $P(R)$ as $R\rightarrow +\infty$. 

\paragraph{Topics of this work} We consider energy subcritical wave equation in the 3-dimensional case 
\[
 \left\{\begin{array}{ll} \partial_t^2 u - \Delta u = \zeta |u|^{p-1}u; \\
 u(\cdot, 0) = u_0;  \\
 u_t (\cdot,0) = u_1.  \end{array}\right.\quad (CP1)
\]
Here $p \in (3,5)$ and $\zeta$ is either $-1$ (defocusing case) or $1$ (focusing case). We show that a weakly non-radiative solution to (CP1) does not necessarily satisfy an estimate similar to \eqref{fall in PR}. In the focusing case $u(x,t) = c_p |x|^{-2/(p-1)}$ is a stationary solution to (CP1) defined for all $(x,t) \in (\Rm^3\setminus\{0\}) \times \Rm$ with a suitable constant $c_p >0$. A simple calculation verifies that $u$ satisfies
\begin{equation}
 \lim_{t\rightarrow +\infty} \int_{|x|>|t|} |\nabla_{x,t} u(x,t)|^2 dx = 0. \label{non radiative temp}
\end{equation}
In addition, the angle between $(c_p r^{-2/(p-1)}, 0)$ and $(1/r,0)$, the single generator of $P(R)$, is a positive constant for all $R>0$. In the defocusing case we are able to give a more interesting example. We construct a solution $u \in C^2(\{(x,t): |x|>|t|\})$ satisfying \eqref{non radiative temp} so that its initial data are ORTHOGONAL to $(1/r,0)$ in $\mathcal{H}_{rad,R}$ for all $R>0$. 

\paragraph{Interpretation} The difference between energy critical and subcritical cases is not as surprising as at the first glance. In the energy critical case, when we consider the solution in an exterior region $\{(x,t): |x|>|t|+R\}$ with a large $R>0$, the linear wave operator dominates the wave propagation of data because this part of solution coincides with a solution with a small energy. In the energy subcritical case, although we still know that the solution in an exterior region $\{(x,t): |x|>|t|+R\}$ with a large $R>0$ carries a small amount of energy, this does not means that linear wave operator dominates the propagation, even in the defocusing case, because the energy space $\dot{H}^1 \times L^2$ is no longer the critical Sobolev space of this Cauchy problem.  
However, if we know the propagation is dominated by the linear part, we may still prove a similar result as in the energy critical case.  In this work we prove that if the initial data $(u_0,u_1)$ is in the critical Sobolev space $\dot{H}^{s_p}\times \dot{H}^{s_p-1}(\Rm^3)$, then any weakly non-radiative solution to (CP1) must coincide with a stationary solution in the exterior region $\{(x,t): |x|>|t|+R\}$. Before we may give the precise statement of our main theorems, we need to define exterior solutions to (CP1) in a suitable way. 

\subsection{Exterior and non-radiative solutions}

Fix $R_0 \geq 0$. Let $\mathcal{H}_{R_0}$ be the space defined in last subsection. We also define a space 
\[
 X(I) = L^{\frac{2p}{p-3}}(I; L^{2p}(\{x: |x|>|t|+R_0\}))
\]
for a time interval $I$. Given any $T>0$ and initial data $(u_0,u_1) \in \mathcal {H}_{R_0}$, we define $\chi_{R_0}$ to be the characteristic function of the exterior region $\{(x,t): |x|>|t|+R_0\}$, use the notation $F(u) = \zeta |u|^{p-1} u$ and consider a transformation from $X([0,T])$ to itself
\[
 \mathbf{T} u = \mathbf{S}_L(t) (u_0,u_1) + \int_0^t \frac{\sin\left((t-\tau)\sqrt{-\Delta}\right)}{\sqrt{-\Delta}} \chi_{R_0} F(u) d\tau, \qquad |x|>|t|+R_0.
\]
Here $\mathbf{S}_L(t)$ is the linear wave propagation operator. In other words, $\mathbf{T} u$ is the solution $\bar{u}$ to linear wave equation $\partial_t^2 \bar{u} - \Delta \bar{u} = \chi_{R_0} F(u)$ in the exterior region with initial data\footnote{One may define $(u_0,u_1)$ in the interior region $\{x: |x|\leq R_0\}$ in any way so that $(u_0,u_1) \in \dot{H}^1(\Rm^3)\times L^2(\Rm^3)$ and consider the solution $\bar{u}$ to linear wave equation $\partial_t^2 \bar{u} - \Delta \bar{u} = \chi_{R_0} F(u)$ with these initial data. By finite speed propagation of wave equation, the choice of $(u_0,u_1)$ does not affect the value of $\bar{u}$ in the exterior region.} $(u_0,u_1)$. We may apply Strichartz estimates and obtain (An almost complete version of Strichartz estimates can be found in Ginibre-Velo \cite{strichartz}. For reader's convenience we put their results in 3-dimensional case in Section \ref{sec: preliminary})
\begin{align*}
 \|\mathbf{T}(u)\|_{X([0,T])} &\leq C_p \left(\|(u_0,u_1)\|_{\mathcal{H}_{R_0}} + \|\chi_{R_0} F(u)\|_{L^1 L^2 ([0,T] \times \Rm^3)}\right)\\
 & \leq C_p \left(\|(u_0,u_1)\|_{\mathcal{H}_{R_0}} + T^{\frac{5-p}{2}} \|\chi_{R_0} F(u)\|_{L^{\frac{2}{p-3}} L^2 ([0,T] \times \Rm^3)}\right)\\
 & = C_p \left(\|(u_0,u_1)\|_{\mathcal{H}_{R_0}} + T^{\frac{5-p}{2}} \|u\|_{X([0,T])}^p\right);
\end{align*}
and 
\begin{align*}
 \|\mathbf{T}(u)-\mathbf{T}(\tilde{u})\|_{X([0,T])} & \leq C_p \|\chi_{R_0} F(u) - \chi_{R_0} F(\tilde{u})\|_{L^1 L^2 ([0,T] \times \Rm^3)} \\
 & \leq p C_p T^{\frac{5-p}{2}} \left(\|u\|_{X([0,T])}^{p-1} + \|\tilde{u}\|_{X([0,T])}^{p-1}\right) \|u-\tilde{u}\|_{X([0,T])}. 
\end{align*}
Thus the transformation $\mathbf{T}$ is a contraction map from the complete metric space $\{u\in X([0,T]): \|u\|_{X([0,T]) < 2C_p \|(u_0,u_1)\|_{\mathcal{H}_{R_0}}}\}$ to itself if 
\[
 T < C(p) \|(u_0,u_1)\|_{\mathcal{H}_{R_0}}^{-\frac{2(p-1)}{5-p}}. 
\]
We may apply a classic fixed-point argument to obtain the local existence and uniqueness of solutions to (CP1) in exterior regions. More details about this type of argument can be found in \cite{loc1, ls}

\begin{definition} [Exterior solutions] \label{def of exterior solution}
We say a function $u$ defined in the exterior region $\{(x,t): |x|<|t|+R_0, -T_-<t<T_+\}$ is an exterior solution to (CP1) with initial data $(u_0,u_1) \in \mathcal{H}_{R_0}$, if $u \in X(I)$ for all bounded closed intervals $I \subset (-T_-,T_+)$, so that 
 \begin{equation}
   u = \mathbf{S}_L(t) (u_0,u_1) + \int_0^t \frac{\sin\left((t-\tau)\sqrt{-\Delta}\right)}{\sqrt{-\Delta}} \chi_{R_0} F(u) d\tau, \quad |x|>|t|+R_0. \label{def of exterior}
\end{equation}
\end{definition}
\begin{remark} \label{exterior is restriction}
We may define initial data $(u_0,u_1)$ in the interior region $\{x: |x|\leq R_0\}$ so that $(u_0,u_1) \in \dot{H}^1(\Rm^3)\times L^2(\Rm^3)$. The right hand side of \eqref{def of exterior} then becomes a function $\tilde{u}$ defined for all $(x,t) \in \Rm^3 \times (-T_-,T_+)$ so that $(\tilde{u}(\cdot,t), \tilde{u}_t(\cdot,t)) \in C((-T_-, T_+); \dot{H}^1(\Rm^3)\times L^2(\Rm^3))$ and $\tilde{u} \in L_{loc}^\frac{2p}{p-3} L^{2p} ((-T_-,T_+)\times \Rm^3)$. Thus an exterior solution defined above is always the restriction of such a function in the exterior region. We also have $F(\tilde{u}), \chi_{R_0} F(u) \in L_{loc}^1 L^2 ((-T_-,T_+) \times \Rm^3)$. 
\end{remark}
\begin{proposition}
 Given any initial data $(u_0,u_1)\in \mathcal{H}_{R_0}$, there is a unique exterior solution to (CP1) with a maximal lifespan $(-T_-,T_+)$ with 
 \[
  |T_-|, |T_+| \geq C(p) \|(u_0,u_1)\|_{\mathcal{H}_{R_0}}^{-\frac{2(p-1)}{5-p}}.
 \]
\end{proposition}

\paragraph{Non-radiative solutions} Now we may define (weakly) non-radiative solutions. 
\begin{definition}[Non-radiative solutions]
 We say a function $u$ defined in the exterior region $\{(x,t): |x|>|t|+R_0\}$ is an $R_0$-weakly non-radiative solution to (CP1) with initial data $(u_0,u_1) \in \mathcal{H}_{R_0}$, if $u$ is an exterior solution to (CP1) defined in Definition \ref{def of exterior solution} so that 
 \[
  \lim_{t\rightarrow \pm \infty} \int_{|x|>|t|+R_0} |\nabla_{x,t} u(x,t) |^2 dx = 0.
 \]
 In particular, we call $u$ a non-radiative solution if $u$ satisfies the conditions above with $R_0=0$. 
\end{definition}

\paragraph{Global existence of defocusing equations} If initial data $(u_0,u_1) \in \dot{H}^1 (\Rm^3) \times L^2(\Rm^3)$ comes with a finite energy
\[
 E = \int_{\Rm^3} \left(\frac{1}{2}|\nabla u_0(x)|^2 +\frac{1}{2}|u_1(x)|^2 + \frac{1}{p+1}|u_0(x)|^{p+1}\right) dx < + \infty,
\]
then we may combine a local theory (with $X(I) = L^{\frac{2p}{p-3}} L^{2p} (I \times \Rm^3)$ defined for the whole space $\Rm^3$ instead of exterior region $\{x: |x|>|t|+R_0\}$) and the energy conservation law to conclude that 

\begin{proposition} \label{global existence whole space}
 Given any $(u_0,u_1) \in (\dot{H}^1 \cap L^{p+1})(\Rm^3) \times L^2(\Rm^3)$, there exists a unique solution to (CP1) in the defocusing case in the whole space-time $\Rm^3 \times \Rm$. 
\end{proposition}
\begin{remark}
 The restriction of a solution given in Proposition \ref{global existence whole space} to the exterior regions is of course an exterior solution to (CP1). Thus given any initial data $(u_0,u_1)\in \mathcal{H}_{R_0}$ with $\|u_0\|_{L^{p+1}(\{x\in \Rm^3: |x|>R_0\})}<+\infty$, the corresponding exterior solution to (CP1) in the defocusing case must be the restriction of a finite-energy solution of (CP1) in the whole space-time to the exterior region thus globally defined in time. 
\end{remark}

\subsection{Main results}

\paragraph{Examples of weakly non-radiative solutions} According to the definition given above, a direct calculation shows that 
\[
 u(x,t) =  \left[\frac{2(p-3)}{(p-1)^2}\right]^{\frac{1}{p-1}} |x|^{-\frac{2}{p-1}}
\]
is an $R$-weakly non-radiative solution to (CP1) in the focusing case for all $R>0$. The angle $\theta$ between $(u(x,0), u_t(x,0))$ and $(1/|x|, 0)$ in the space $\mathcal{H}_{rad, R}$ is determined by
\[
 \cos \theta = \frac{\left\langle(|x|^{-2/(p-1)},0), (|x|^{-1},0)\right\rangle_{\mathcal{H}_{rad,R}}}{\|(|x|^{-2/(p-1)},0)\|_{\mathcal{H}_{rad,R}} \|(|x|^{-1},0)\|_{\mathcal{H}_{rad,R}}} = \frac{1}{2}\sqrt{(5-p)(p-1)} \in (0,1).
\]
Thus $\theta \in (0,\pi/2)$ is a positive angle independent of $R$. Please note that $u(x,t)$ is NOT a non-radiative solution to (CP1) because the initial data $u(x,0) \notin \dot{H}^1 (\Rm^3)$. In the defocusing case we may give a more interesting example:

\begin{theorem} \label{main 1}
 There exists a radial $C^2$ solution $u$ to (CP1) in the defocusing case defined in the exterior region $\{(x,t): |x|>|t|\}$ with nonzero initial data $(0, u_1)$ so that $u$ is an $R$-weakly non-radiative solution to (CP1) for any $R>0$.  
\end{theorem}

\begin{remark}
 This is clear that $\langle(0,u_1), (1/|x|, 0) \rangle_{\mathcal{H}_{rad, R}} = 0$ for all $R>0$, i.e. initial data $(0,u_1)$ is orthogonal to $(1/|x|,0)$ in $\mathcal{H}_{rad,R}$. But the example we give is not a non-radiative solution to (CP1), because $u_1 \notin L^2(\Rm^3)$ although $u_1 \in L^2(\{x: |x|>R\})$ for all $R>0$. Please see section \ref{sec: example} for more details. 
\end{remark}

\paragraph{Stationary solutions} Next we consider stationary solutions to (CP1) with similar asymptotic behaviour to $1/|x|$, the single generator of $P(R)$ in dimension 3. The focusing case (Proposition \ref{stationary focusing}) has been considered in Shen \cite{shen2}. The defocusing case (Proposition \ref{stationary defocusing}) is discussed in Section \ref{sec: stationary defocusing}. 

\begin{proposition} \label{stationary focusing}
 The elliptic equation $-\Delta U(x) = |U|^{p-1} U(x)$ has a radial solution $U^+ \in C^\infty (\Rm^3 \setminus \{0\})$ so that $U^+ \notin \dot{H}^{s_p}(\Rm^3)$ and 
 \begin{align}
  &\left|U^+(x)-1/|x|\right| \lesssim |x|^{2-p},& &|\nabla U^+(x)| \lesssim 1/|x|^2, & &|x|\gg 1.&  \label{asymptotic of stationary}
 \end{align}
\end{proposition}

\begin{proposition} \label{stationary defocusing}
 The elliptic equation $-\Delta U(x) = -|U|^{p-1} U(x)$ has a radial solution $U^- \in C^\infty (\{x: |x|>R_-\})$ so that $U^-$ has the same asymptotic behaviour as in \eqref{asymptotic of stationary} and the blow-up 
 \[
   \lim_{|x|\rightarrow (R_-)^+} U^-(x) = + \infty, \qquad R_->0.
 \]
\end{proposition}
\begin{remark}
 The function $U^+ \notin \dot{H}^1(\Rm^3)$, otherwise we might combine $U^+ \in \dot{H}^1(\Rm^3)$ with the asymptotic behaviour of $U^+$ to obtain $u \in \dot{W}^{1,\frac{3(p-1)}{p+1}} (\Rm^3) \hookrightarrow \dot{H}^{s_p}(\Rm^3)$. This contradicts the already known fact $U^+ \notin \dot{H}^{s_p}$. The defocusing case is similar.  According to Lemma \ref{radial estimate}, a radial $\dot{H}^1 (\Rm^3)$ function $u(x)$ can never blow up when $|x|$ approaches a positive number. Thus $U^-$ is not the restriction of any radial $\dot{H}^1(\Rm^3)$ function. 
\end{remark}
\noindent We may define a family of radial stationary solutions to (CP1) by rescaling 
\begin{align*}
 U_C^+(x,t) = U_C^+ (x) = \left\{\begin{array}{ll} C^{-\frac{2}{p-3}}U^+(x/C^{\frac{p-1}{p-3}}), & \hbox{if}\quad C>0; \\ 0, & \hbox{if}\quad C=0; \\ 
 -|C|^{-\frac{2}{p-3}}U^+(x/|C|^{\frac{p-1}{p-3}}), & \hbox{if}\quad C<0; \end{array}\right. 
\end{align*}
and
\begin{align*}
 U_C^-(x,t) = U_C^- (x) = \left\{\begin{array}{lll} C^{-\frac{2}{p-3}}U^-(x/C^{\frac{p-1}{p-3}}), & |x|>C^{\frac{p-1}{p-3}} R_-, & \hbox{if}\quad C>0; \\ 0, & |x|>0, & \hbox{if}\quad C=0; \\ 
 -|C|^{-\frac{2}{p-3}}U^-(x/|C|^{\frac{p-1}{p-3}}), & |x|>|C|^{\frac{p-1}{p-3}} R_-, & \hbox{if}\quad C<0. \end{array}\right. 
\end{align*}
The solutions $U_C^+(x,t)$ are $R$-weakly non-radiative for any $R>0$. The solution $U_C^-(x,t)$ is $R$-weakly non-radiative for any $R>|C|^{\frac{p-1}{p-3}} R_-$. We recall the behaviour of $U^\pm$ near infinity, conduct a simple calculation and obtain
\[
  \left|U_C^\pm (x,t) - C/|x|\right| \lesssim |x|^{2-p}, \quad |x| \gg 1. 
\]
Please note that these solutions are NOT non-radiative unless $C=0$ because we have $U_C^\pm (x) \notin \dot{H}^1 (\Rm^3)$ for $C\neq 0$. 

\paragraph{Weakly non-radiative solution in $\dot{H}^{s_p}\times \dot{H}^{s_p-1}$} The second main result of this work is that any weakly non-radiative solutions to (CP1) in the critical Sobolev space must coincide with a stationary solution as given above. For convenience we define $\mathcal{H}_R^{s_p}$ to be the space of restrictions of all $(\dot{H}^1\cap \dot{H}^{s_p}) \times (L^2 \cap \dot{H}^{s_p-1})$ functions in the exterior region $\{x\in \Rm^3: |x|>R\}$. 

\begin{theorem} \label{main 2}
Assume $R>0$. Let $u$ be a radial $R$-weakly non-radiative solution to (CP1) with initial data $(u_0,u_1)\in \mathcal{H}_R^{s_p}$. Then there exist a constant $C$, so that $u(x,t) \equiv U_C^\pm (x,t)$ in the exterior region $\{(x,t): |x|>|t|+R\}$.  In the defocusing case, the constant $C$ above satisfies $|C| < (R/R_-)^{\frac{p-3}{p-1}}$.
\end{theorem}

\noindent Since a non-radiative solution to (CP1) is $R$-weakly non-radiative for all $R>0$, we may apply Theorem \ref{main 2} with $R\rightarrow 0^+$ and utilize the fact $U_C^\pm (x,0) \notin \dot{H}^1 (\Rm^3)$ for $C\neq 0$ to obtain

\begin{corollary} \label{rad solution}
 If $u$ is a radial non-radiative solution to (CP1) with initial data $(u_0,u_1) \in (\dot{H}^1(\Rm^3) \cap \dot{H}^{s_p}(\Rm^3)) \times (L^2(\Rm^3)\cap \dot{H}^{s_p-1}(\Rm^3))$, then $u(x,t)\equiv 0$ for all $(x,t)$ with $|x|>|t|$.
\end{corollary}

\begin{remark}
Initial data $(u_0,u_1) \in \mathcal{H}_R$ are contained in the space $\mathcal{H}_R^{s_p}$ as long as they coincide with some $\dot{H}^{s_p}\times \dot{H}^{s_p-1}$ data near the infinity. In fact, if $(u_0,u_1) \in \dot{H}^1(\Rm^3) \times L^2(\Rm^3)$ and $(u'_0, u'_1) \in \dot{H}^{s_p}(\Rm^3) \times \dot{H}^{s_p-1}(\Rm^3)$ satisfy\footnote{More precisely, the identity holds in the sense of distribution.}
\[
 (u_0(x), u_1(x))=(u'_0(x), u'_1(x)), \qquad |x|>R_1.
\]
then we may write $(u_0,u_1) = (1-\mathbf{P}_{2R_1})(u_0,u_1) + \mathbf{P}_{2R_1}(u'_0,u'_1)$. Here $\mathbf{P}$ is a center cut-off operator as defined in Lemma \ref{center cutoff}. Since $\mathbf{P_{2R_1}}$ is a bounded operator from $\dot{H}^s$ to itself for $s \in [-1,1]$, we have  $\mathbf{P}_{2R_1}(u'_0,u'_1) \in \dot{H}^{s_p} \times \dot{H}^{s_p-1}$ and $(1-\mathbf{P}_{2R_1})(u_0,u_1) \in \dot{H}^1 \times L^2$. The latter is also compactly supported, thus 
\[
 (1-\mathbf{P}_{2R_1})(u_0,u_1) \in \dot{W}^{1,\frac{3(p-1)}{p+1}} \times L^\frac{3(p-1)}{p+1} \hookrightarrow \dot{H}^{s_p} \times \dot{H}^{s_p-1}.
\]
In summary we have $(u_0,u_1) \in \dot{H}^{s_p} \times \dot{H}^{s_p-1}$.
\end{remark}

\section{Preliminary Results} \label{sec: preliminary}

\subsection{Technical Lemma}

\begin{lemma}[center cut-off operator] \label{center cutoff}
Given a constant $s \in [-1,1]$. Fix a smooth radial cut-off function $\phi: \Rm^3 \rightarrow [0,1]$ satisfying
\[
 \phi(x) = \left\{\begin{array}{ll} 0, & |x|\leq 1/2; \\ 1, & |x|\geq 1. \end{array}\right.
\]
We define an operator $\mathbf{P}_R f = \phi(x/R) f$. Here $R>0$ is a positive constant. Then $\mathbf{P}_R$ is a bounded operator from $\dot{H}^s(\Rm^3)$ to $\dot{H}^s(\Rm^3)$, whose operator norm $\|\mathbf{P}_R\|_{\dot{H}^s \rightarrow \dot{H}^s}$is independent of $R$. In addition, if $f \in \dot{H}^s(\Rm^3)$, then $\|\mathbf{P}_R f\|_{\dot{H}^s(\Rm^3)} \rightarrow 0$ as $R\rightarrow +\infty$.  
\end{lemma}
\begin{proof}
 We first show $\mathbf{P}_R$ is a bounded operator. By dilation it suffices to consider the case $R=1$. A basic calculation shows that $\mathbf{P}_1$ is a bounded operator from $\dot{H}^1$ to itself, and from $L^2$ to itself. An interpolation then gives the boundedness for all $s \in [0,1]$. By duality the operator $\mathbf{P}_1$ is also bounded for $s \in [-1,0]$. Next we show the limit $\|\mathbf{P}_R f\|_{\dot{H}^s} \rightarrow 0$ as $R\rightarrow +\infty$. By the uniform boundedness of $\mathbf{P}_R$, it suffices to show that this limit holds for $f$ in a dense subset of $\dot{H}^s$. Now we may finish the proof by observing that this limit clearly holds for $f \in C_0^\infty (\Rm^3)$. 
\end{proof}
\begin{lemma}[See Lemma 3.2 of \cite{km}] \label{radial estimate}
 Let $u \in \dot{H}^s(\Rm^3)$ be a radial function, $1/2<s<3/2$. Then we have the following pointwise estimate
\[
 |u(x)| \lesssim_s \frac{\|u\|_{\dot{H}^s(\Rm^3)}}{|x|^{3/2-s}}.  
\]
\end{lemma}
\begin{lemma}\label{radial estimate outside} 
 If $u \in \dot{H}^1(\Rm^3)$ be a radial function. Then we have
 \[
  |u(x)| \lesssim |x|^{-1/2} \left(\frac{1}{4\pi}\int_{|y|>|x|} |\nabla u(y)|^2 dy\right)^{1/2}. 
 \]
\end{lemma}
\begin{proof}
 If $u$ is smooth and compactly supported, we have 
\begin{align*}
 |u(x)| = \left|\int_{|x|}^\infty u_r(r) dr \right| & \leq \left(\int_{|x|}^\infty \frac{1}{r^2} dr\right)^{1/2}\left(\int_{|x|}^\infty r^2 |u_r(r)|^2 dr\right)^{1/2} \\
 & = |x|^{-1/2} \left(\frac{1}{4\pi}\int_{|y|>|x|} |\nabla u(y)|^2 dy\right)^{1/2}
\end{align*}
A standard smooth approximation and cut-off technique then deals with the general case. 
\end{proof}
\begin{lemma} \label{uw relationship}
 Let $u \in \dot{H}^1(\Rm^3)$ be radial. Then given any $R>0$, the one-variable function $w(r) = ru(r)$ satisfies
 \[
  \int_{R}^\infty |w_r(r)|^2 dr = \frac{1}{4\pi} \int_{|x|>R} |\nabla u(x)|^2 dx - R |u(R)|^2. 
 \]
\end{lemma}
\begin{proof}
We may apply the identity $w_r(r) = r u_r(r) + u(r)$ and calculate
\begin{align*}
 \int_{R}^{R'} |w_r(r)|^2 dr & = \int_R^{R'} \left(r^2 u_r + 2r u_r u + u^2\right) dr \\
 & = \int_R^{R'} |u_r|^2 r^2 dr + \int_R^{R'} \partial_r (ru^2) dr \\
 & = \frac{1}{4\pi} \int_{R<|x|<R'} |\nabla u(x)|^2 dx + R' |u(R')|^2 - R |u(R)|^2.
\end{align*}
Finally we make $R'\rightarrow +\infty$ and finish the proof. Here we need to use the following fact(see Lemma A.7 of \cite{shen2}): If $u$ is a radial $\dot{H}^1$ function, then $R' |u(R')|^2\rightarrow 0$ as $R'\rightarrow \infty$. 
\end{proof}
\noindent This immediately gives
\begin{corollary} \label{w non radiative}
 Let $u$ be a radial $R$-weakly non-radiative solution to (CP1). Then the function $w(r,t) = ru(r,t)$ satisfies
\[
 \lim_{t\rightarrow \pm \infty} \int_{|t|+R}^\infty (|w_r(r,t)|^2 + |w_t(r,t)|^2) dr = 0.
\] 
\end{corollary}
\subsection{Local theory in critical Sobolev spaces}

\begin{proposition} [Generalized Strichartz estimates, see \cite{strichartz}] Let $2\leq q_1,q_2 \leq \infty$, $2\leq r_1,r_2 < \infty$ and $\rho_1,\rho_2,s\in \Rm$ be constants with 
\begin{align*}
 &1/q_i+1/r_i \leq 1/2, \; i=1,2; & &1/q_1+3/r_1=3/2-s+\rho_1;& &1/q_2+3/r_2=1/2+s+\rho_2.&
\end{align*}
Assume that $u$ is the solution to the linear wave equation
\[
 \left\{\begin{array}{ll} \partial_t u - \Delta u = F(x,t), & (x,t) \in \Rm^3 \times [0,T];\\
 u|_{t=0} = u_0 \in \dot{H}^s(\Rm^3); & \\
 \partial_t u|_{t=0} = u_1 \in \dot{H}^{s-1}(\Rm^3). &
 \end{array}\right.
\]
Then we have
\begin{align*}
 \left\|\left(u(\cdot,T), \partial_t u(\cdot,T)\right)\right\|_{\dot{H}^s \times \dot{H}^{s-1}} & +\|D_x^{\rho_1} u\|_{L^{q_1} L^{r_1}([0,T]\times \Rm^3)} \\
 & \leq C\left(\left\|(u_0,u_1)\right\|_{\dot{H}^s \times \dot{H}^{s-1}} + \left\|D_x^{-\rho_2} F(x,t) \right\|_{L^{\bar{q}_2} L^{\bar{r}_2} ([0,T]\times \Rm^3)}\right).
\end{align*}
Here the coefficients $\bar{q}_2$ and $\bar{r}_2$ satisfy $1/q_2 + 1/\bar{q}_2 = 1$, $1/r_2 + 1/\bar{r}_2 = 1$. The constant $C$ does not depend on $T$ or $u$. 
\end{proposition}
 
\noindent Combining suitable Strichartz estimates with a fixed-point argument, we have the following scattering theory with small data in the critical Sobolev space.
 
 \begin{proposition}[Scattering with small initial data] \label{scattering with small initial data}
 There exists a constant $\delta = \delta(p)>0$, so that if the initial data satisfy $\|(u_0,u_1)\|_{\dot{H}^{s_p} \times \dot{H}^{s_p-1}} < \delta$, then the corresponding solution $u$ to (CP1) exists globally in time and scatters with $\|(u(\cdot,t), u_t(\cdot,t))\|_{\dot{H}^{s_p} \times \dot{H}^{s_p-1}} < 2\delta$. 
\end{proposition}

\section{Examples of weakly non-radiative solutions} \label{sec: example}
In this section we prove that energy subcritical wave equation in the defocusing case admits weakly non-radiative solutions that are orthogonal to $(r^{-1},0)$ in the energy space, i.e. we prove Theorem \ref{main 1}.
 
\paragraph{Reduction to ODE} The $C^2$ solution we construct is in the form of $u(x,t) = |x|^{-2/(p-1)} f(t/|x|)$ with initial data $(u_0,u_1) = (0,a|x|^{-2/(p-1)-1})$. Here $f$ is a $C^2$ function defined on $(-1,1)$ and $a>0$ is a parameter. Please note $u_1 \notin L^2(\Rm^3)$ but $u_1 \in L^2 (\{x\in \Rm^3: |x|>R\})$ for any $R>0$. We may use polar coordinates and the notation $\beta =2/(p-1)$ to write $u(r,t) = r^{-\beta} f(t/r)$ for convenience. A straightforward calculation shows
\begin{align*}
 u_{tt} & = r^{-\beta-2} f''(t/r); \\
 u_r &  = -\beta r^{-\beta-1} f(t/r) - t r^{-\beta-2} f'(t/r); \\
 u_{rr} & = \beta(\beta+1) r^{-\beta-2} f(t/r) + (2\beta+2) t r^{-\beta-3} f'(t/r) + t^2 r^{-\beta-4} f''(t/r);\\
 \Delta u & = u_{rr} + (2/r) u_r = \beta(\beta-1) r^{-\beta-2} f(t/r) + 2\beta t r^{-\beta-3}f'(t/r) + t^2 r^{-\beta-4} f''(t/r);\\
 |u|^{p-1} u & = r^{-\beta-2} |f(t/r)|^{p-1} f(t/r).
\end{align*}
We plug these in the defocusing wave equation $\partial_t^2 u - \Delta u = -|u|^{p-1}u$ and obtain 
\[
 r^{-\beta-2} \left[\left(1-\frac{t^2}{r^2}\right)f''\left(\frac{t}{r}\right) - 2\beta \cdot \frac{t}{r}f'\left(\frac{t}{r}\right) + \beta(1-\beta)f\left(\frac{t}{r}\right) +\left|f\left(\frac{t}{r}\right)\right|^{p-1} f\left(\frac{t}{r}\right) \right] = 0.
\]
Therefore $f$ satisfies the ordinary differential equation 
 \begin{equation}
  \left\{\begin{array}{ll} (1-x^2) f''(x) - 2 \beta x f'(x) + \beta(1-\beta) f(x) + |f(x)|^{p-1} f(x) = 0, & x \in (-1,1);\\
  f(0) = 0,\; f'(0) = a. & \end{array} \right. \label{ode details 1}
 \end{equation}
Each solution $f$ to \eqref{ode details 1} gives a solution $u(x,t) = |x|^{-2/(p-1)} f(t/|x|) $ to the defocusing wave equation defined on $\{(x,t): |t|<|x|\}$. Some useful properties of solutions to the initial value problem \eqref{ode details 1} are summarized in the following proposition. We postpone its proof until Section \ref{sec: ode1} of this work since it is irrelevant to our main topics. 

\begin{proposition} \label{ode}
 Let $\beta \in (1/2,1)$, $\gamma>0$ and $p >1$ be constants. The solutions to the ordinary differential equation 
 \[
  \left\{\begin{array}{ll} (1-x^2) f''(x) - 2 \beta x f'(x) + \gamma f(x) + |f(x)|^{p-1} f(x) = 0, & x \in (-1,1);\\
  f(0) = 0,\; f'(0) = a; & \end{array} \right.
 \]
 satisfy the following properties 
 \begin{itemize}
  \item[(i)] The solutions $f(x)$ are classic solutions defined for all $x \in (-1,1)$; i.e. $f \in C^2((-1,1))$. 
  \item[(ii)] We have continuous dependence of $f(x)$ on initial value $a$ up to the endpoints, i.e. we may define $f(x)$ at $x = \pm 1$ so that $f(x)$ becomes a continuous function of $(x,a) \in [-1,1]\times \Rm$. In addition, we have a uniform upper bound $|f(x)| \lesssim_{\beta, \gamma, p} |a|$. 
  \item[(iii)] there exists a continuous function $G = G(a)$ so that the behaviour of $f'(x)$ near endpoints is given by
  \[
    \left|f'(x) - G(1-x^2)^{-\beta} - \frac{1}{2\beta} \left[\gamma f(1)+ |f(1)|^{p-1} f(1)\right]\right| \lesssim_{\beta,\gamma,p} (|a|+|a|^{p}) (1-|x|)^{1-\beta}.
  \]
  In addition, we have $G=f(1)=0$ if and only if $a=0$.
  \item[(iv)] There are infinitely many positive initial values $a>0$, so that the solution $f$ satisfies 
  \[
   \sup_{x\in (-1,1)} |f'(x)| < +\infty.
  \]
 \end{itemize}
\end{proposition}

\paragraph{Weakly non-radiative solutions} Now let us choose a positive parameter $a$ as in part (iv) and consider the solution $u(x,t) = |x|^{-2/(p-1)} f(t/|x|)$. According to part (ii) of Proposition \ref{ode} we have a uniform upper bound $|u(x,t)| \lesssim |x|^{-2/(p-1)}$ for all $|x|>t$. A simple calculation shows that $u \in X(\Rm) =  L^{\frac{2p}{p-3}}(\Rm; L^{2p}(\{x: |x|>|t|+R\}))$ for any $R>0$. Thus $u$ is always an exterior solution to (CP1) in the exterior region $\{(x,t): |x|>|t|+R\}$. Our choice of initial value $a$ guarantees that both $f$ and $f'$ are bounded, therefore we have the following estimates for any $r=|x|>t$:
\begin{align*}
  |u_r| \lesssim &  r^{-2/(p-1)-1} |f(t/r)| + |t| r^{-2/(p-1)-2} |f'(t/r)| \lesssim r^{-2/(p-1)-1}; \\
  |u_t| \lesssim &  r^{-2/(p-1)-1} |f'(t/r)|  \lesssim r^{-2/(p-1)-1}.
\end{align*}
Thus
\begin{align*}
 \int_{|x|>|t|} \left(|\nabla u(x,t)|^2 + |u_t(x,t)|^2\right) dx & = 4\pi \int_{|t|}^\infty (|u_r(r,t)|^2+|u_t(r,t)|^2) r^2 dr\\
 & \lesssim \int_{|t|}^\infty r^{-4/(p-1)} dr \lesssim |t|^{-\frac{5-p}{p-1}}. 
\end{align*}
This vanishes as $|t|\rightarrow \infty$. As a result, $u$ is an $R$-weakly non-radiative solution to (CP1) for any $R>0$.

\section{Weakly non-radiative solutions in $\dot{H}^{s_p}\times \dot{H}^{s_p-1}$}

In this section we give a proof of Theorem \ref{main 2}. The general idea comes from Duyckaerts-Kenig-Merle\cite{se}, In the author's previous work \cite{shen2} the same idea is used to deal with soliton-like minimal blow-up solutions $v$ obtained via the compactness-rigidity argument, whose trajectory $\{(v(\cdot,t), v_t (\cdot,t)): t\in \Rm\}$ is pre-compact in both spaces $\dot{H}^1\times L^2$ and $\dot{H}^{s_p} \times \dot{H}^{s_p-1}$. A soliton-like minimal blow-up solution is clearly a special case of non-radiative solutions. In this work we improve the argument so that it works for all $R$-weakly non-radiative solution with initial data in $\mathcal{H}_R^{s_p}$.  

\subsection{Asymptotic behaviour of non-radiative solutions}

The lemmata in this subsection describe behaviour of weakly non-radiative radial solutions $u(x,t)$ to (CP1) with initial data in $\mathcal{H}_R^{s_p}$ when $x$ is sufficiently large. 
large. 
\begin{lemma} \label{initial decay}
 Let $u$ be a radial, $R_0$-weakly non-radiative solution to (CP1) with initial data $(u_0,u_1)\in \mathcal{H}_{R_0}^{s_p}$. Then given any $\eps>0$, there exists a large radius $R_* =R_* (\eps,u)>0$, so that the inequality $|u(r,t)| \leq \eps r^{-2/(p-1)}$ holds for all $r>\max\{|t|+R_0, R_*\}$.  
\end{lemma}
\begin{proof}
 Given any small positive constant $\eps< \eps(p)$, we may choose a large radius $R$ so that 
 \[
  \|\mathbf{P}_R (u_0,u_1)\|_{\dot{H}^{s_p} \times \dot{H}^{s_p-1}}  <\eps.
 \]
Here $\mathbf{P}_R$ is the center cut-off operator defined in Lemma \ref{center cutoff}. Let $u^{(R)}$ be the solution to (CP1) with initial data $\mathbf{P}_R (u_0,u_1)$. Scattering theory with small initial data (Proposition \ref{scattering with small initial data}) then guarantees that $u^{(R)}$ is globally defined in time and satisfies 
 \[
   \left\|(u^{(R)}(\cdot,t), u_t^{(R)}(\cdot,t))\right\|_{\dot{H}^{s_p}(\Rm^3)\times \dot{H}^{s_p-1}(\Rm^3)} < 2\eps, \quad \forall t\in \Rm.
 \]
 By finite speed of propagation and Lemma \ref{radial estimate}, we have $|u(r,t)| = |u^{(R)}(r,t)| \leq 2 C_p \eps r^{-2/(p-1)}$ for all $(r,t)$ with $r\geq |t|+R$. These points $(r,t)$ are exactly those in the darker grey region of figure \ref{figure regions}. We still need to deal with $(r,t)$ so that $|t|+R_0<r<|t|+R$. In fact for these $(r,t)$ the function $w(r,t) = ru(r,t)$ satisfies 
 \begin{align*}
  |w(r,t)| & \leq |w(|t|+R, t)| + \int_{r}^{|t|+R} |w_r(r',t)| dr' \\
  & \leq 2C_p \eps (|t|+R)^{1-2/(p-1)} + (|t|+R-r)^{1/2} \left(\int_{r}^{|t|+R} |w_r(r',t)|^2 dr'\right)^{1/2}\\
  & \leq 2C_p \eps (r+R-R_0)^{1-2/(p-1)} + (R-R_0)^{1/2} \left(\int_{|t|+R_0}^\infty |w_r(r',t)|^2 dr'\right)^{1/2}\\
  & \leq 2C_p \eps (r+R-R_0)^{1-2/(p-1)} + (R-R_0)^{1/2} \left(\frac{1}{4\pi}\int_{|x|>|t|+R_0} |\nabla u(x,t)|^2 dx\right)^{1/2}.
 \end{align*}
In the final step we apply Lemma \ref{uw relationship}. The latter term in the final line above has an upper bound independent of $(r,t)$ by our non-radiative assumption and Remark \ref{exterior is restriction}. Thus there exists a sufficiently large radius $R_*>0$, so that if $\max\{R_*, |t|+R_0\} < r < |t|+R$, i.e. the point $(r,t)$ is in the lighter grey region of figure \ref{figure regions}, then $|w(r,t)|\leq 3C_p \eps r^{1-2/(p-1)}\Rightarrow |u(r,t)| \leq 3C_p \eps r^{-2/(p-1)}$. Combining this with the case $r\geq |t|+R$, we finish the proof.
\end{proof}

\begin{figure}[h]
 \centering
 \includegraphics[scale=0.75]{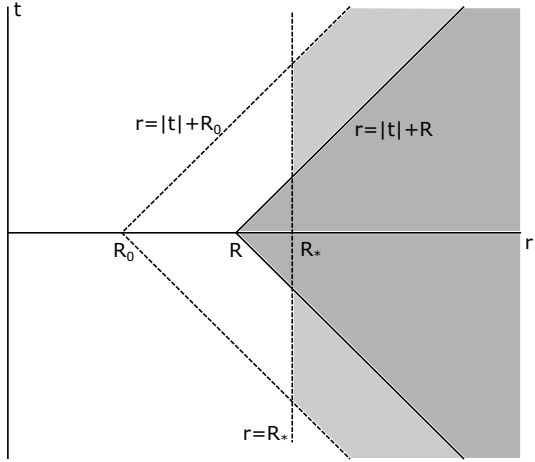}
 \caption{Illustration of regions in the proof of Lemma \ref{initial decay}} \label{figure regions}
\end{figure}

\begin{lemma} \label{characteristic lines plus non radiative} 
Assume that $R \geq R_0>0$. Let $u$ be a radial, $R_0$-weakly non-radiative solution to (CP1) satisfying ($\beta \geq 2/(p-1)$)
 \[
  |u(r,t)| \leq \eps r^{-\beta}, \qquad r>\max\{|t|+R_0, R\}.
 \]
Then we have 
\begin{itemize}  
 \item[(a)] The function $v_+(r,t) = (\partial_t-\partial_r)(ru)$ satisfies the identity 
 \[
  v_+(r,t) = -\zeta \int_{t}^\infty (t'-t+r)|u|^{p-1} u(t'-t+r,t') dt'
 \]
 for almost everywhere $r>|t|+R_0$.
 \item[(b)] The inequality $|\partial_r (ru)|, |ru_t| \leq 2\eps^{p} r^{2-p\beta}$ hold for almost everywhere $r>\max\{|t|+R_0, R\}$.
\end{itemize}
\end{lemma}
\begin{proof}
 The function $w(r,t) = ru(r,t)$ solves the one-dimensional wave equation $w_{tt} - w_{rr} = \zeta r|u|^{p-1} u$. Therefore  $v_+(r,t) = w_t(r,t)-w_r(r,t)$ satisfies
 \begin{align}
  \frac{d}{dt}v_+(t-t_1+r,t) & = \zeta (t-t_1+r) |u|^{p-1} u(t-t_1+r,t),& & t>t_1,\; r>|t_1|+R_0; \nonumber\\
  v_+(t_2\!-\!t_1\!+\!r,t_2) \!-\! v_+(r,t_1) &= \int_{t_1}^{t_2}\! \zeta (t\!-\!t_1\!+\!r) |u|^{p-1} u(t\!-\!t_1\!+\!r,t) dt,& & t_2>t_1,\; r>|t_1|+R_0. \label{variation of v plus}
 \end{align}
 Our assumption on decay of $u$ implies that the absolute value of integrand satisfies 
 \[
  (t-t_1+r) |u(t-t_1+r,t)|^{p} \leq \eps^p (t-t_1+r)^{1-p\beta}, \qquad t > \max\{t_1,R-R_0\}, \, r>|t_1|+R_0.
 \]
 Thus if we fix $t_1$, then the right hand side integral of \eqref{variation of v plus} converges uniformly for all $r>|t_1|+R_0$ as $t_2\rightarrow +\infty$. Combining this uniform convergence, our non-radiative assumption and Corollary \ref{w non radiative}, we may make $t_2\rightarrow +\infty$ and obtain an identity
 \[
  v_+(r,t_1) = -\int_{t_1}^{\infty} \zeta (t-t_1+r) |u|^{p-1} u(t-t_1+r,t) dt, \quad \hbox{in}\; L_{loc}^2(\{r: r>|t_1|+R_0\}).
 \]
 This proves part (a). If $r>\max\{|t|+R_0,R\}$, we may use the conclusion of part (a), and plug the decay assumption $u(r,t)\leq \eps r^{-\beta}$ in the right hand integral to conclude 
 \[
  |v_+(r,t)| \leq  2\eps^p r^{2-p\beta}, \quad a.e.\; r>\max\{|t|+R_0,R\}.
 \]
 We may prove a similar inequality about $v_-(r,t) = (\partial_t + \partial_r)(ru)$ in the same manner. Combining these two inequalities we finish the proof.
\end{proof}
\begin{lemma} \label{induction beta}
Assume that $R_0>0$ and $R\geq \max\{1,R_0\}$. Let $u$ be a radial, $R_0$-weakly non-radiative solution to (CP1) satisfying 
 \[
  |u(r,t)| \leq \eps r^{-\beta}, \qquad r>\max\{|t|+R_0,R\}
 \]
 for a sufficiently small constant $\eps < \eps_0(p)$ and $\beta \in [2/(p-1), 3/p]$. Then there exists a large radius $R_1 = R_1(p,R)$ and a small constant $\kappa = \kappa(p)>0$ so that
 \[
  |u(r,t)| \leq \eps r^{-\beta - \kappa}, \qquad r>\max\{|t|+R_0, R_1\}.
 \]
 \end{lemma}
 \begin{proof}
 Let us define ($n=0,1,2,\cdots$)
 \[
  a_n = \sup \{r^{\beta}|u(r,t)|: t \in \Rm, r>\max\{|t|+R_0, 2^n R\}\} \leq \eps. 
 \]
 Thus we have $|u(r,t)|\leq a_n r^{-\beta}$ for all $(r,t)$ with $r>\max\{|t|+R_0, 2^n R\}$. Given any $(r,t)$ with $r>\max\{|t|+R_0, 2^n R\}$, we may utilize Lemma \ref{characteristic lines plus non radiative} and verify that $w(r,t) = ru(r,t)$ satisfies  
  \begin{align*}
   |w(r,t)| & \leq |w(r/2,0)| + |w(r,0)-w(r/2,0)| + |w(r,t)-w(r,0)|\\
   &\leq  |w(r/2,0)| + \int_{r/2}^r |w_r(r',0)| dr' + \left|\int_{0}^{t} w_t(r,t') dt'\right|\\
   & \leq (r/2)^{1-\beta} a_{n-1} + 4 a_{n-1}^p r^{3-p\beta}.
  \end{align*}
  Thus we have
  \begin{align*}
   a_n  & = \sup \{r^{\beta-1}|w(r,t)|: t\in \Rm, r>\max\{|t|+R, 2^n R\} \}\\
    & \leq \sup \left\{(1/2)^{1-\beta} a_{n-1} + 4 a_{n-1}^p r^{2-(p-1)\beta}: t\in \Rm, r>\max\{|t|+R, 2^n R\}\right\} \\
   & \leq [(1/2)^{1-3/p} + 4 \eps^{p-1}] a_{n-1} \leq  \lambda a_{n-1}. 
  \end{align*}
  Here we may choose an arbitrary constant $\lambda = \lambda(p) \in ((1/2)^{1-3/p}, 1)$ and then determine $\eps = \eps(p)$ accordingly. Therefore we have $a_n \leq \eps \lambda^n$. As a result, given any $(r,t)$ with $r>\max\{|t|+R_0,R\}$, we may choose $n = \max\{n\in {\mathbb Z}: 2^n R < r\}$ and find an upper bound 
  \begin{align*}
   r^\beta |u(r,t)| \leq a_n = \eps \lambda^n \leq \eps \lambda^{\log_2 (r/2R)} = \eps (2R)^{\log_2(1/\lambda)} r^{-\log_2 (1/\lambda)}
  \end{align*}
  Thus we may choose an arbitrary constant $\kappa = \kappa(p)  \in (0, \log_2 (1/\lambda))$ and determine $R_1 = R_1(R,p)$ accordingly so that the inequality $|u(r,t)| \leq \eps r^{-\beta - \kappa}$ holds for all $(r,t)$ with $r>\max\{|t|+R_0, R_1\}$.
 \end{proof}
 \begin{proposition} \label{asymptotic behaviour for match}
  Assume that $R_0>0$ and $R\geq \max\{1,R_0\}$. Let $u$ be a radial, $R_0$-weakly non-radiative solution to (CP1) satisfying 
 \[
  |u(r,t)| \leq \eps r^{-2/(p-1)}, \qquad r>\max\{|t|+R_0,R\}
 \]
 for a sufficiently small positive constant $\eps< \eps_0(p)$. Then there exists two constants $C \in \Rm$ and $R'>1$ so that 
 \begin{align*}
  &|u(r,t)-C/r| \lesssim r^{2-p},&  &\forall \; r> \max\{|t|+R_0,R'\};& \\
  &|u_r(r,t)+C/r^2|+|u_t(r,t)| \lesssim r^{1-p},& &\forall \;a.e.\; r> \max\{|t|+R_0,R'\}.&
 \end{align*}
 \end{proposition} 
 \begin{proof}
  First of all, we gain better decay estimates of $u$ by induction. Application of Lemma \ref{induction beta} multiple times leads to a finite sequence $R<R_1<R_2<\cdots < R_n$ with $2/(p-1)+(n-1)\kappa \leq 3/p < 2/(p-1) + n\kappa$ and
  \begin{align*}
  &|u(r,t)| \leq \eps r^{-2/(p-1)-j\kappa}, \; \forall r>\max\{|t|+R_0, R_j\},& & j=1,2,\cdots, n.
  \end{align*}
 For convenience we define $\beta = 2/(p-1) + n \kappa > 3/p$ and apply Lemma \ref{characteristic lines plus non radiative}. 
 \[
  |w_r(r,t)|, \;|w_t(r,t)| \leq 2\eps^p r^{2-p\beta}, \quad \forall \; a.e. \; r>\max\{|t|+R_0, R_n\}.
 \]
 Since $2-p\beta<-1$, the function $w(r,t)$ converges as $r\rightarrow +\infty$ for all fixed $t\in \Rm$. In addition, this limit is independent of $t$ because of the decay estimate of $w_t$. Therefore there exists a constant $C$, so that 
 \[
  \left|w(r,t) - C\right| \lesssim r^{3-p\beta}, \quad r>\max\{|t|+R_0, R_n\}.
 \]
 It immediately follows that $|w(r,t)|\lesssim 1 \Rightarrow |u(r,t)|\lesssim r^{-1}$. We then apply Lemma \ref{characteristic lines plus non radiative} again and obtain 
 \[
  |w_r(r,t)|+|w_t(r,t)| \lesssim r^{2-p}\; a.e.\Rightarrow \left|w(r,t) - C\right| \lesssim r^{3-p}, \quad \forall r>\max\{|t|+R_0, R_n\}.
 \]
 Finally we rewrite these inequalities in term of $u$ and finish the proof.
 \end{proof}
 
 \subsection{Coincidence of Non-radiative Solutions}
 In this subsection we show that two weakly non-radiative radial solutions with the same asymptotic behaviour as $r\rightarrow \infty$ must be exactly the same. 
 
 \begin{lemma} \label{match further}
  Assume that $R'>R_0>0$. Let $u$ and $\tilde{u}$ be two radial, $R_0$-weakly non-radiative solutions to (CP1) so that $u(r,t) = \tilde{u}(r,t)$ if $r> |t|+R'$. Then the identity $u(r,t) = \tilde{u}(r,t)$ also holds if $r> |t|+R_0$.
 \end{lemma}
 
 \begin{figure}[h]
 \centering
 \includegraphics[scale=0.75]{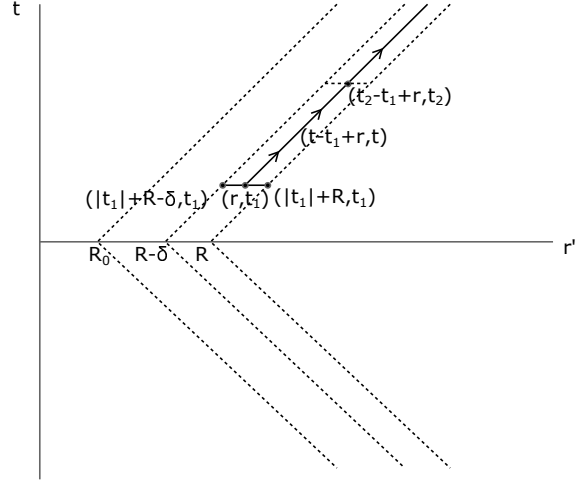}
 \caption{Illustration of integral path} \label{figure integral path}
\end{figure}
 
 \begin{proof}
 Let us define 
 \[
  R = \min \{r' \geq R_0: u(r,t) = \tilde{u}(r,t) \; \hbox{if} \; r>|t|+r'\} \leq R'.
 \]
 It suffices to show $R=R_0$. If $R>R_0$, we define a function $g$ for $\delta \in (0,R-R_0)$. Functions $v_\pm, w$ below be defined as in the proof of Lemma \ref{characteristic lines plus non radiative}, $\tilde{v}_\pm, \tilde{w}$ are derived from $\tilde{u}$ in the same manner.
 \begin{align*}
  g(\delta) & = \sup_{t\in \Rm} \left\{\int_{|t|+R-\delta}^{\infty} |w_r(r,t)-\tilde{w}_r(r,t)|^2 + |w_t(r,t)-\tilde{w}_t(r,t)|^2 dr\right\} \\
  & = \sup_{t \in \Rm} \left\{\int_{|t|+R-\delta}^{|t|+R} |w_r(r,t)-\tilde{w}_r(r,t)|^2 + |w_t(r,t)-\tilde{w}_t(r,t)|^2 dr\right\}
 \end{align*}
We always have $g(\delta)<+\infty$ by our non-radiative assumption, Remark \ref{exterior is restriction} and Lemma \ref{uw relationship}. By the same argument as in Lemma \ref{characteristic lines plus non radiative}, the following identity holds for any time $t_2>t_1$
 \begin{align}
  & [v_+(t_2 \!-\! t_1\!+\! r,t_2) - \tilde{v}_+(t_2\!-\!t_1\!+\!r, t_2)] - [v_+(r,t_1)-\tilde{v}_+(r,t_1)] \label{L2 identity w tilde}\\
  & \quad = \int_{t_1}^{t_2}\! \zeta (t \! -\! t_1 \!+\! r) [|u|^{p-1} u(t \!-\! t_1 \!+\! r,t) - |\tilde{u}|^{p-1}\tilde{u}(t\!-\!t_1\!+\!r,t)] dt, \;\; \hbox{in}\; L_r^2 (J(t_1,\delta)). \nonumber
 \end{align}
For convenience we use the notation $J(t_1,\delta) = [|t_1|+R-\delta, |t_1|+R]$. Please see figure \ref{figure integral path} for an illustration of the integral path involved. By considering the limits of both sides of \eqref{L2 identity w tilde} in the space $L_r^2 (J(t_1,\delta))$ as $t_2 \rightarrow +\infty$, we obtain an identity 
\begin{equation} \label{L2 identity infinity}
 \tilde{v}_+(r,t_1) - v_+(r,t_1) = \int_{t_1}^{\infty}\! \zeta (t \! -\! t_1 \!+\! r) [|u|^{p-1} u(t \!-\! t_1 \!+\! r,t) - |\tilde{u}|^{p-1}\tilde{u}(t\!-\!t_1\!+\!r,t)] dt.
\end{equation}
The limit of the left hand side is relatively easy. We only need to recall the non-radiative assumption, apply Corollary \ref{w non radiative} and obtain
\[
 \lim_{t_2\rightarrow +\infty} \left(\|v_+(t_2 - t_1 + r,t_2)\|_{L_r^2(J(t_1,\delta))} +\|\tilde{v}_+(t_2 - t_1 + r, t_2)\|_{L_r^2(J(t_1,\delta))}\right) = 0.
\]
 In order to evaluate the limit of the right hand side we first give upper bounds of $u$, $\tilde{u}$ as well as $w-\tilde{w}$. We recall Remark \ref{exterior is restriction}, our non-radiative assumption and apply Lemma \ref{radial estimate outside} to obtain
\begin{align*}
 &|u(r,t)| \leq M r^{-1/2}, \; r>|t|+R_0; & &M = \sup_{t\in \Rm} \left(\frac{1}{4\pi}\int_{|x|>|t|+R_0} |\nabla u(x)|^2 dx\right)^{1/2} < +\infty.& \\
 &|\tilde{u}(r,t)| \leq \tilde{M} r^{-1/2}, \; r>|t|+R_0; & &\tilde{M} = \sup_{t\in \Rm} \left(\frac{1}{4\pi}\int_{|x|>|t|+R_0} |\nabla \tilde{u}(x)|^2 dx\right)^{1/2} < +\infty.&
\end{align*}
We may also find an upper bound of $w-\tilde{w}$ at $(t-t_1+r,t)$ with $r\in J(t_1,\delta)$ and $t>t_1$
 \begin{align*}
  \left|(w-\tilde{w})(t-t_1+r,t)\right| &\leq \left|(w-\tilde{w})(t-t_1+r+\delta,t)\right| + \int_{t-t_1+r}^{t-t_1+r+\delta} \left|w_r(r',t) - \tilde{w}_r(r',t)\right| dr'\\
    & \leq \delta^{1/2} \left(\int_{t-t_1+r}^{t-t_1+r+\delta} \left|w_r(r',t) - \tilde{w}_r(r',t)\right|^2 dr'\right)^{1/2} \leq \delta^{1/2} g(\delta)^{1/2}. 
 \end{align*}
 Here $t-t_1+r+\delta \geq t-t_1+(|t_1|+R-\delta)+\delta \geq |t|+R$, thus $|(w-\tilde{w})(t-t_1+r+\delta,t)| = 0$. Similarly we have $t-t_1+r\geq |t|+R-\delta$ thus the integral of $|w_r-\tilde{w}_r|^2$ is dominated by $g(\delta)$. Combining these two estimates we may find an upper bound of the integrand in the right hand side of \eqref{L2 identity w tilde} as below. Please note that all the functions are evaluated at $(t-t_1+r,t)$ unless specified otherwise.
\begin{align*}
 \left|\zeta (t \! -\! t_1 \!+\! r) [|u|^{p-1} u - |\tilde{u}|^{p-1}\tilde{u}]\right| & \leq p[|u|+|\tilde{u}|]^{p-1} |w -\tilde{w}| \\
 & \leq p (M+\tilde{M})^{p-1} \delta^{1/2} g(\delta)^{1/2} (t-t_1+r)^{-\frac{p-1}{2}}.
\end{align*}
Here $\frac{p-1}{2} > 1$. Thus as $t_2 \rightarrow \infty$, the integral in the right hand side of \eqref{L2 identity w tilde} converges to that of \eqref{L2 identity infinity} uniformly for all $r \in J(t_1,\delta)$. This immediately gives the convergence in $L_r^2 (J(t_1,\delta))$. Next we substitute the integrand in \eqref{L2 identity infinity} by its upper bound given above and obtain
\begin{align*}
 \left|\tilde{v}_+(r,t_1) \!-\! v_+(r,t_1)\right| & \leq \int_{t_1}^{\infty} p (M+\tilde{M})^{p-1} \delta^{1/2} g(\delta)^{1/2} (t\!-\!t_1\!+\!r)^{-\frac{p-1}{2}} dt\\
 & \leq \frac{2p}{p-3} (M+\tilde{M})^{p-1} \delta^{1/2} g(\delta)^{1/2} r^{-\frac{p-3}{2}}
\end{align*}
Thus 
\begin{align*}
\int_{|t_1|+R-\delta}^{|t_1|+R} \left|\tilde{v}_+(r,t_1) \!-\! v_+(r,t_1)\right|^2 dr \leq & \int_{|t_1|+R-\delta}^{|t_1|+R} \frac{4p^2}{(p-3)^2} (M+\tilde{M})^{2p-2} \delta g(\delta) R_0^{3-p}dr\\
 = & C(p,M,\tilde{M},R_0) \delta^2 g(\delta).
\end{align*}
Similarly we have
 \[
  \int_{|t_1|+R-\delta}^{|t_1|+R} |\tilde{v}_-(r,t_1) - v_-(r,t_1)|^2 dr \leq C(p,M,\tilde{M},R_0) \delta^2 g(\delta).
 \]
Since these inequalities hold for all $t_1 \in \Rm$, we have
\[
 2g(\delta) = \sup_{t\in \Rm} \int_{|t|+R-\delta}^{|t|+R} \left(|v_-(r,t) - \tilde{v}_-(r,t)|^2 + |v_+(r,t) - \tilde{v}_+(r,t)|^2 \right) dr \leq 2C(p,M,\tilde{M},R_0) \delta^2 g(\delta).
\]
This means $g(\delta) = 0$ for sufficiently small $\delta >0$, which implies that $w(r,t) = \tilde{w}(r,t)$ for all $(r,t)$ with $r>|t|+R-\delta$, thus gives a contradiction. 
 \end{proof}
 \begin{proposition} \label{match two non radiative}
  Let $u$ and $\tilde{u}$ be two radial, $R_0$-weakly non-radiative solutions to (CP1). In addition, there exists a large radius $R'>\max\{R_0,1\}$ and a constant $C>0$ so that
  \begin{align*}
   &|u(r,t)|+|\tilde{u}(r,t)| \leq \frac{C}{r}, & &r>\max\{|t|+R_0,R'\};&\\
   &\lim_{r\rightarrow +\infty} \left|ru(r,t)-r\tilde{u}(r,t)\right| = 0,& &\forall t\in \Rm.&
  \end{align*}
  Then $u(r,t) \equiv \tilde{u}(r,t)$ for all $(r,t)$ with $r> |t|+R_0$.
 \end{proposition}
 \begin{proof}
  Without loss of generality we assume $R_0>0$. Otherwise we first prove the identity for $r>|t|+R$ with small positive numbers $R>0$ and then let $R\rightarrow 0^+$. We first apply Lemma \ref{characteristic lines plus non radiative} on both $u$ and $\tilde{u}$ to obtain ($v_\pm, w$ are defined as in the proof of Lemma \ref{characteristic lines plus non radiative}, $\tilde{v}_\pm, \tilde{w}$ are derived from $\tilde{u}$ in the same manner)
  \begin{equation}
    v_+(r,t) - \tilde{v}_+(r,t) = \zeta \int_{t}^\infty (t'-t+r)\left[|\tilde{u}|^{p-1} \tilde{u}(t'-t+r,t') - |u|^{p-1} u(t'-t+r,t')\right] dt' \label{difference of v and tilde}
  \end{equation}
  Now let us assume $|\tilde{u}(r,t)-u(r,t)|\leq C r^{-\beta}$ for some constant $\beta \geq 1$ and all $r>\max\{R_0+|t|, R'\}$. (This holds for $\beta =1$) Then we immediately have
  \begin{align*}
    \left|v_+(r,t) - \tilde{v}_+(r,t)\right| &\leq \int_t^\infty (t'-t+r) \cdot p [C(t'-t+r)^{-1}]^{p-1} |u(t'-t+r,t')-\tilde{u}(t'-t+r,t')| dt'\\
    & \leq \frac{p}{p-3+\beta} C^p r^{-(p-3+\beta)}. 
  \end{align*}
  Similarly we may prove an inequality regarding $v_-$ and $\tilde{v}_-$. By $v_\pm = w_t \mp w_r$ and $\tilde{v}_\pm = \tilde{w}_t \mp \tilde{w}_r$ we have
 \[
  |w_r (r,t)-\tilde{w}_r(r,t)| \leq \frac{p}{p-3+\beta} C^p r^{-(p-3+\beta)}, \;\forall \, a.e.\, r>\max\{R_0+|t|,R'\}
 \]
 Combining this with our assumption on the limit of $w-\tilde{w}$ as $r\rightarrow \infty$, we have
 \begin{align*}
  |w(r,t) - \tilde{w}(r,t)| &\leq \frac{p}{(p-3+\beta)(p-4+\beta)} C^p r^{-(p-4+\beta)}, & &r>\max\{R_0+|t|,R'\};\\
  |u(r,t) - \tilde{u}(r,t)| &\leq \frac{p}{(p-2)(p-3)} C^p r^{-(p-3+\beta)}, & &r>\max\{R_0+|t|,R'\}.
 \end{align*}
 Without loss of generality we may assume (otherwise we may enlarge $R'$)
 \begin{align*}
  \frac{p}{(p-2)(p-3)} C^{p-1} (R')^{-\frac{p-3}{2}} < 1.
 \end{align*}
 Thus we have $|u(r,t) - \tilde{u}(r,t)| \leq C r^{-\frac{p-3}{2}-\beta}$ if $r>\max\{R_0+|t|,R'\}$. By induction we have $|u(r,t) - \tilde{u}(r,t)| \leq C r^{-\frac{p-3}{2}n-\beta}$ for all $n=0,1,2,\cdots$ and $r>\max\{R', |t|+R_0\}$. This means $u(r,t) \equiv \tilde{u}(r,t)$ for all $r>\max\{R', |t|+R_0\}$. Finally we apply Lemma \ref{match further} to conclude that $u(r,t) \equiv \tilde{u}(r,t)$ for all $r>|t|+R_0$.
 \end{proof}
 
 \subsection{Proof of Theorem \ref{main 2}}
 
 Now let us assume that $u$ is a radial, $R_0$-weakly non-radiative solution to (CP1) with initial data $(u_0,u_1) \in \mathcal{H}_{R_0}^{s_p}$. By lemma \ref{initial decay}, given $\eps >0$, there exists a large radius $R_*=R_*(u,\eps)$, so that $|u(r,t)| \leq \eps r^{-2/(p-1)}$ holds for all $r>\max\{|t|+R_0,R_1\}$. This enable us to apply Proposition \ref{asymptotic behaviour for match} and obtain two constants $C\in \Rm$, $R'>1$ so that
 \begin{align*}
  |u(r,t)-C/r| \lesssim r^{2-p},\quad \hbox{if}\; r> \max\{|t|+R_0,R'\}.
 \end{align*} 
 We claim that the constant $C$ also satisfies $|C|^{\frac{p-1}{p-3}} R_- < R_0$ in the defocusing case. If this were false, i.e. $R_{C}^- \doteq |C|^{\frac{p-1}{p-3}} R_- \geq R_0 > 0$, then we might apply Proposition \ref{match two non radiative} on $u$ and $U_{C}^-$ in the region $\{(x,t): |x|>|t|+ R_C^- +\eps\}$ with an arbitrary $\eps>0$. We obtain $u(x,t) = U_C^- (x)$ if $|x|>|t|+ R_C^- +\eps$. By making $\eps \rightarrow 0^+$ we have $u(x,t) = U_C^- (x)$ for all $(x,t)$ with $|x|>|t|+R_C^-$. Thus $u_0(x) = u(x,0) = U_C^-(x)$ blows up when $|x| \rightarrow (R_C^-)^+$. This gives a contradiction, thanks to Lemma \ref{radial estimate}. Finally we are able to apply Proposition \ref{match two non radiative} on $u$ and $U_{C}^\pm$ in the region $\{(x,t): |x|>|t|+ R_0\}$ to conclude that $u(x,t) = U_C^\pm (x)$ whenever $|x|>t+R_0$ and finish the proof. 
  
 \section{Appendix A: Ordinary Differential Equations} \label{sec: ode1}
In this section we consider the ordinary differential equation 
\begin{equation}
  \left\{\begin{array}{ll} (1-x^2) f''(x) - 2 \beta x f'(x) + \gamma f(x) + |f(x)|^{p-1} f(x) = 0, & x \in (-1,1);\\
  f(0) = 0,\; f'(0) = a; & \end{array} \right. \label{ode detail}
\end{equation}
and prove Proposition \ref{ode}. If we use the notation $\lesssim$ in the proof below, then the implicit constant depends on $\beta, \gamma, p$ unless specified otherwise. 

\subsection{Global existence} 
 Classic theory of ordinary differential equations guarantees that the solution $f$ is $C^2$ and defined in a maximal interval $(-\delta, \delta)$ for some $\delta \in (0,1]$. In order to verify $\delta = 1$ we only need to show that $f(x)$ and $f'(x)$ are both bounded in the interval $(-\delta, \delta)$ if $\delta < 1$. This can be done by a semi-conservation law. We may multiply both sides of equation \eqref{ode detail} by $(1-x^2)^{2\beta-1}f'(x)$ and obtain
\[
 \frac{d}{dx}\left[\frac{1}{2}(1-x^2)^{2\beta} |f'(x)|^2 + (1-x^2)^{2\beta-1} P(f(x))\right] = -2(2\beta-1)x(1-x^2)^{2\beta -2} P(f(x)). 
\]
Here the potential $P$ is defined by
\[
 P(y) = \frac{\gamma}{2}|y|^2 + \frac{1}{p+1}|y|^{p+1} \geq 0. 
\]
The derivative above is nonpositive if $x>0$ and nonnegative if $x<0$. Thus we always have
\begin{equation}
  \frac{1}{2}(1-x^2)^{2\beta} |f'(x)|^2 + (1-x^2)^{2\beta-1} P(f(x)) \leq \frac{a^2}{2}. \label{semi energy}
\end{equation}
This immediately gives the boundedness of $f'(x)$ and $f(x)$, as well as the continuous dependence of $f(x)$ and $f'(x)$ on parameter $a$, as long as $x$ is away from the endpoints $\pm 1$. Before we conclude this subsection, we also give another semi-conservation law for future use. We may multiply the original equation by $f'(x)$ and obtain 
\[
 \frac{d}{dx}\left[\frac{1}{2}(1-x^2)|f'(x)|^2 + P(f(x))\right] = (2\beta-1)x |f'(x)|^2. 
\]
Therefore we can find a lower bound regarding $f(x)$ and $f'(x)$. 
\begin{equation} \label{semi energy lower}
 \frac{1}{2}(1-x^2)|f'(x)|^2 + P(f(x)) \geq \frac{a^2}{2}, \quad \forall x\in (-1,1). 
\end{equation}
\subsection{Continuity of $f(x)$ at the endpoints} 
Now let us consider the behaviour of $f(x)$ when $x \rightarrow 1^-$. The behaviour of $f(x)$ when $x\rightarrow -1^+$ is similar because $f$ is an odd function. First of all, the inequality \eqref{semi energy} implies 
\begin{align*}
 |f'(x)| \leq |a| (1-x)^{-\beta}, \quad x>0.
\end{align*}
Since $\beta<1$, we know the limit $\displaystyle f(1) \doteq \lim_{x\rightarrow 1^-} f(x)$ is well-defined. In addition, if we fix $x_0 \in (0,1)$, then we have
\begin{itemize}
 \item The function $f$ is a continuous function of $(x,a) \in [0,x_0]\times \Rm$. 
 \item The upper bound of $f'(x)$ given above also implies
\begin{equation} \label{variation near 1}
 \sup_{x_1, x_2\in [x_0,1]}|f(x_1)-f(x_2)| \leq \int_{x_0}^1 |f'(y)| dy \lesssim_\beta |a|(1-x_0)^{1-\beta}.
\end{equation}
\end{itemize}
These immediately give the continuity of $f$ on $(x,a) \in [0,1]\times \Rm$. We may also combine \eqref{semi energy} and \eqref{variation near 1} (with $x_0 = 1/2$) to give an upper bound
\begin{equation}
 \max_{x\in [0,1]} |f(x)| \lesssim |a|. \label{uniform upper bound f}
\end{equation}
\subsection{Asymptotic behaviour of $f'(x)$ at endpoints}
In order to investigate the asymptotic behaviour of $f'(x)$ as $x\rightarrow 1^-$, we calculate (We always assume $x\geq 0$ in this subsection, the property of $f(x)$ for negative $x$ can be obtained by symmetry)
\begin{equation}
 \frac{d}{dx} \left[(1-x^2)^\beta f'(x)\right] = (1-x^2)^{\beta-1} \left[(1-x^2) f''(x) - 2\beta x f'(x)\right] = - (1-x^2)^{\beta-1} P'(f(x)). \label{G0}
\end{equation}
Here we use the equation \eqref{ode detail} again. Since $\left|(1-x^2)^{\beta-1} P'(f(x))\right| \lesssim (1-x)^{\beta-1} P'(|a|)$ is integrable in $[0,1]$, we have a well-defined limit $\displaystyle G \doteq \lim_{x \rightarrow 1^-} (1-x^2)^\beta f'(x)$ with 
\begin{align}
 G - (1-x^2)^\beta f'(x) & = -\int_x^1(1-y^2)^{\beta-1} P'(f(y)) dy; \label{G01}\\
  \left|G - (1-x^2)^\beta f'(x)\right| & \lesssim \int_x^1 (1-y)^{\beta-1} P'(|a|) dy \lesssim (1-x)^{\beta} P'(|a|); \label{variation of beta f prime}\\
  \left|f'(x) - (1-x^2)^{-\beta} G\right| & \leq C_0 P'(|a|), \quad x\in [0,1). \label{variation of f prime}
\end{align}
We may combine \eqref{variation of beta f prime} with the fact that $f'(x_0)$ depends continuously on parameter $a$ for a fixed $x_0 \in (0,1)$ to conclude that $G$ is a continuous function of $a$. Now let us have a more careful look at the asymptotic behaviour of $f'(x)$ near $1$. According to \eqref{variation near 1} and \eqref{uniform upper bound f}, we have
\begin{align*}
 |P'(f(x))-P'(f(1))| \lesssim P''(|a|) |f(x)-f(1)| \lesssim P''(|a|)|a|(1-x)^{1-\beta} \lesssim P'(|a|) (1-x)^{1-\beta}.
\end{align*}
We combine this with \eqref{G01} and obtain
\begin{align*}
 \left|G- (1-x^2)^\beta f'(x) + \int_x^1 (1-y^2)^{\beta-1}P'(f(1)) dy\right|& \leq \int_x^1 (1-y^2)^{\beta-1}\left|P'(f(y)-P'(f(1))\right| dy\\
 & \lesssim  \int_x^1 (1-y^2)^{\beta-1}(1-y)^{1-\beta} P'(|a|) dy\\
 & \lesssim P'(|a|)(1-x).
\end{align*}
Thus we have
\begin{equation} \label{asymptotic f prime}
 \left|f'(x) - G(1-x^2)^{-\beta} - \frac{1}{2\beta} P'(f(1))\right| \lesssim P'(|a|) (1-x)^{1-\beta}. 
\end{equation}
Finally we claim that $G$ and $f(1)$ can never be zero at the same time unless $a=0$. In fact, if $G = f(1) = 0$, the estimate above implies that $f'(x) \rightarrow 0$ as $x \rightarrow 1^-$.  Our semi-conservation law \eqref{semi energy lower} then gives $a=0$. 
\subsection{Extreme values of $f$}
We prove part (iv) of Proposition \ref{ode} by considering the extreme values of $f$ on (0,1). We have
\begin{proposition} \label{extreme values}
 The solution $f$ to ordinary differential equation \eqref{ode detail} satisfies 
 \begin{itemize}
  \item [(a)] Given $a>0$, there are finitely many points $x \in (0,1)$ so that $f'(x) = 0$. All of these are local maxima or minima. We use the notation $N(a)$ for the number of local extreme points.
  \item [(b)] Let $a = a_0$ be a positive parameter so that $\displaystyle G \doteq \lim_{x\rightarrow 1^-} (1-x^2)^\beta f'(x) \neq 0$. Then $N(a)$ is a constant in a small neighbourhood of $a_0$.  
  \item [(c)] When $a>0$ is large, we have a lower bound $N(a) \gtrsim a^{\frac{p-1}{p+1}}$. 
 \end{itemize}
\end{proposition}
We temporarily postpone the proof of Proposition \ref{extreme values} and first show why part (iv) of Proposition \ref{ode} is a direct consequence of it. 
\paragraph{Proof of part (iv)} First of all, the approximation formula
\[
 \left|f'(x) - G(1-x^2)^{-\beta} - \frac{1}{2\beta} P'(f(1))\right| \lesssim P'(|a|) (1-|x|)^{1-\beta}
\]
given in part (iii) implies that $\displaystyle \sup_{x \in (-1,1)} |f'(x)|<+\infty$ is equivalent to $G=0$. If there were only finite number of $a$'s so that $G=0$, then $N(a)$ would be a constant for all sufficiently large $a>0$, by part (b) of Proposition \ref{extreme values}. However, this contradicts with part (c). 
\begin{proof}[Proof of Proposition \ref{extreme values}]
 Let us start by part (a). If $f'(x) = 0$ for some $x \in (0,1)$, then $f''(x) \neq 0$. Otherwise we have $f(x)=0$ thus $f\equiv 0$. Thus $x$ must be either a maximum, if $f''(x)<0$, or a minimum, if $f''(x)>0$. This also implies that all these extreme points are isolated from each other. Thus it suffices to show these points can not accumulate around the endpoints $0,1$. The case of $x=0$ is trivial since $f \in C^2$ and we have assumed $f'(0) = a >0$. The case $x=1$ can be dealt with by the approximation formula of $f'(x)$ near $x=1$ given in \eqref{asymptotic f prime}. Please note that at least one of $G$ and $f(1)$ is nonzero. Now let us prove part (b). Let $x_1<x_2<\cdots<x_n$ be all extreme points of $f$ in $(0,1)$ when $a=a_0$. We can always choose a sufficiently small positive constant $\eps$ so that
 \[
  0 < x_1-\eps < x_1+\eps < x_2-\eps < x_2+\eps < \cdots < x_n-\eps < x_n +\eps < 1-\eps < 1
 \]
 satisfy
 \begin{align*}
  &\inf_{x\in [0,x_1-\eps]} f'(x) > 0;& &\inf_{x\in [x_i+\eps, x_{i+1}-\eps]} |f'(x)|>0, \; i=1,2,\cdots, n-1;& &\inf_{x\in [x_n+\eps, 1-\eps]} |f'(x)|>0;&
 \end{align*}
 \begin{align*}
  &\inf_{x \in [x_i-\eps,x_i+\eps]} |f''(x)|>0,\; f'(x_i\!-\!\eps) f'(x_i\!+\!\eps) < 0, \; i=1,\cdots,n;& &[1\!-\!(1\!-\!\eps)^2]^{-\beta} |G| > C_0 P'(|a|).&
 \end{align*}
 Here the constant $C_0=C_0(\beta,\gamma,p)$ is the one in \eqref{variation of f prime}. The final inequality above guarantees that $f'(x) \neq 0$ for all $x \in [1-\eps,1)$. The continuous dependence of $f(x), f'(x), f''(x)$ (away from $x =\pm 1$) and $G$ on parameter $a$ then guarantees that all the inequalities above also hold for parameters $a$ in a small neighbourhood of $a_0$. It immediately follows that there is exactly one extreme point in each interval $[x_i-\eps,x_i+\eps]$ but none elsewhere. Finally we prove part (c). We consider an interval $I = [y_1,y_2]\subseteq (0,1/2)$ which does not contain an extreme point or zero of $f(x)$. According to the semi-conservation law \eqref{semi energy lower}, either $|f'(x)|> |a|/2$ or $|f(x)|> |a|^{2/(p+1)}/2$ holds for any $x\in (-1,1)$. Thus 
\begin{equation}
 |I| \leq \left|\{x\in I: |f'(x)|>a/2\}\right| + \left|\{x\in I: |f(x)|>a^{2/(p+1)}/2\}\right|. \label{upper bound of length}
\end{equation}
Next we observe the following facts 
\begin{itemize} 
 \item None of $f(x)$ and $f'(x)$ may change its sign in $I$ by our assumption on the interval $I$;
 \item The inequalities $|f(x)| \lesssim a^{2/(p+1)}$ and $|f'(x)| \lesssim a$ hold for all $x \in [0,1/2]$ by semi-conservation law \eqref{semi energy}.
\end{itemize} 
These help to give upper bounds of the terms in the right hand side of \eqref{upper bound of length}: 
 \begin{align*}
  \left|\{x\in I: |f'(x)|>a/2\}\right| \leq \frac{2}{a} \left|\int_{y_1}^{y_2} f'(x) dx \right| = \frac{2}{a} \left|f(y_1)-f(y_2)\right| \lesssim a^{-\frac{p-1}{p+1}};
 \end{align*}
 and (we utilize \eqref{G0} below)
 \begin{align*}
   \left|\{x\in I: |f(x)|>a^{2/(p+1)}/2\}\right| & \leq \left|\{x\in I : (1-x^2)^{\beta-1} |P'(f(x))| > a^{2p/(p+1)}/2^p \}\right| \\
   & \lesssim a^{-\frac{2p}{p+1}} \left|\int_{y_1}^{y_2} (1-x^2)^{\beta-1} P'(f(x)) dx \right|\\
   & = a^{-\frac{2p}{p+1}} \left|(1-y_1^2)^\beta f'(y_1) - (1-y_2^2)^\beta f'(y_2)\right|\\
   & \lesssim a^{-\frac{p-1}{p+1}}.
 \end{align*}
 In summary we have $|I| \lesssim a^{-\frac{p-1}{p+1}}$, i.e. there is a constant $C=C(\beta,\gamma,p)$, so that $|I| \leq C a^{-\frac{p-1}{p+1}}$. Because there is at least one extreme point between any two zeros of $f$, it immediately follows that any interval $I \subset (0,1/2)$ with $|I| > 2C a^{-\frac{p-1}{p+1}}$ must contain at least an extreme point. This finishes the proof.
\end{proof}

\section{Appendix B: Elliptic Equation} \label{sec: stationary defocusing}

In this section we consider the elliptic equation $-\Delta U = \zeta |U|^{p-1} U$. This gives stationary solutions to (CP1). The case with $\zeta = +1$ has been discussed in the author's previous work \cite{shen2}. We still need to deal with the case $\zeta = -1$ and prove Proposition \ref{stationary defocusing}. We will follow roughly the same argument as in the case $\zeta = +1$, thus we omit some details of proof and focus on the difference of these two cases. 
\paragraph{Transformation to one-dimensional case} We define $z(|x|) = |x| U(x)$ and consider the equation $z(r)$ satisfies 
\[
 z''(r) = \frac{|z(r)|^{p-1} z(r)}{r^{p-1}}, \qquad r>0.
\] 
The existence of $z$ near infinity (i.e. for $r\in [R,\infty)$ with a large $R>0$) with prescribed asymptotic behaviour $z(+\infty) = 1$ and $z'(+\infty) = 0$ then follows a fixed point argument. We then solve $z(r)$ backward by a standard ODE theory. This argument is exactly the same as in the focusing case. Please see Section 9 of \cite{shen2} for details. The only difference is that $z(r)$ can no longer be defined for all $r>0$ but blows up at $r = R_->0$ in the current setting. We will discuss this blow-up phenomenon in details. 

\paragraph{Monotonicity} When $r$ is sufficiently large, we know $z(r)>1$, $z'(r)<0$, $z''(r)>0$. A continuity argument then verifies that all these inequalities still hold in the whole lifespan of $z$. Thus $z(r)$ is either defined for all $r>0$ or blows up to $+\infty$ at some point $r=R_-$. We will show that the former can never happen. 

\paragraph{Iteration of lower bounds} Because $z(r)>1$, we have $z''(r) = |z(r)|^{p-1} z(r) / r^{p-1} \geq r^{-(p-1)}$. We may integrate, use $z'(+\infty) = 0$ and obtain 
\[
 z'(r) \leq - \frac{1}{p-2} r^{-(p-2)}. 
\]
We integrate again, use $z(+\infty) = 1$ and obtain 
\[
 z(r) \geq \frac{1}{(p-2)(p-3)} r^{-(p-3)} + 1 \geq \frac{1}{(p-2)(p-3)} r^{-(p-3)}. 
\]
We may iterate this argument and obtain a family of lower bounds 
\[
 z(r) \geq \frac{r^{-\beta_k}}{c_k}, \quad k=0,1,2,\cdots
\]
Here the coefficients are defined by induction 
\begin{align*}
 &\beta_{k+1} = p \beta_k + (p-3),& &c_{k+1} = (p\beta_k+p-3)(p\beta_k+p-2)c_k^p, & &(\beta_0, c_0)= (0,1).&
\end{align*}
We may give an explicit formula $\beta_k = \frac{(p-3)(p^k-1)}{p-1}$. We also have
\begin{align*}
 \ln c_{k+1} = p \ln c_k + \ln (p \beta_k +p-3) + \ln(p \beta_k + p-2) & \leq p \ln c_k + 2(k+1)\ln p\\
\Rightarrow \ln c_{k+1} + (k+3)\ln p & \leq p \left[\ln c_k + (k+2)\ln p\right] \\
\Rightarrow \ln c_{k} + (k+2) \ln p & \leq 2 p^k \ln p. 
\end{align*}
Thus we have 
\[
 \ln z(r) \geq \beta_k \ln (1/r) - \ln c_k \geq \left(\frac{p-3}{p-1}\ln \frac{1}{r} - 2\ln p\right) (p^k-1) -2\ln p, \quad \forall k=0,1,2,\cdots
\]
This implies that $z(r)$ can not be defined for $r< p^{-2(p-1)/(p-3)}$ otherwise the inequality above fails when $k\rightarrow +\infty$.

\end{document}